\documentclass[11pt]{amsart}

\usepackage{amssymb}
\usepackage{latexsym}
\usepackage{amsmath}
\usepackage{amsthm}
\usepackage{amsfonts}
\usepackage{color}
\usepackage{pdfsync}
\usepackage[usenames,dvipsnames]{pstricks}
\usepackage{epsfig}
\usepackage{pst-grad} 
\usepackage{pst-plot} 

\newcommand{\Lq}{L^q(\R^n)}
\newcommand{\R}{\mathbb{R}}

\theoremstyle{plain}
\newtheorem{defi}{Definition}[section]
\newtheorem{prop}[defi]{Proposition}
\newtheorem{teo}[defi]{Theorem}
\newtheorem{cor}[defi]{Corollary}
\newtheorem{lema}[defi]{Lemma}

\theoremstyle{definition}

\theoremstyle{remark}

\numberwithin{equation}{section}

\begin{document}

\title[Fractional Decay for Zero Order Heat Equations]{Fractional Decay Bounds for Nonlocal Zero Order Heat Equations.}

\author{E. Chasseigne}
\address{
Emmanuel Chasseigne:
Laboratoire de Math\'ematiques et Physique Th\'eorique, CNRS UMR 7350,
F\'ed\'eration Denis Poisson FR2964, Universit\'e Francois Rabelais, Parc de Grandmont, 37200 Tours, France.
\newline {\tt Emmanuel.Chasseigne@lmpt.univ-tours.fr}
}

\author{P. Felmer}
\address{
Patricio Felmer:
Departamento de Ingenier\'\i a Matem\'atica and CMM (UMI 2807 CNRS), Universidad de Chile, Casilla 170 Correo 3, Santiago, Chile.
\newline {\tt pfelmer@dim.uchile.cl}
}

\author{J. D. Rossi}
\address{
J. D. Rossi: Departamento de An\'{a}lisis Matem\'{a}tico,
Universidad de Alicante, Ap. correos 99, 03080, Alicante, Spain
and Departamento de Matem\'{a}tica, FCEyN UBA, Ciudad
Universitaria, Pab 1 (1428), Buenos Aires, Argentina. 
\newline {\tt julio.rossi@ua.es}
}

\author{E. Topp}
\address{
Erwin Topp:
Departamento de Ingenier\'\i a Matem\'atica (UMI 2807 CNRS), Universidad de Chile, Casilla 170, Correo 3, Santiago, Chile.
and Laboratoire de Math\'ematiques et Physique Th\'eorique (CNRS UMR 7350), Universit\'e Fran\c{c}ois Rabelais,
Parc de Grandmont, 37200, Tours, France. 
\newline {\tt etopp@dim.uchile.cl} 
}


\date{\today}

\begin{abstract}
In this paper we obtain bounds for the decay rate
for solutions to the nonlocal problem $\partial_t u(t,x) = \int_{\R^n} J(x,y)[u(t,y) - u(t,x)] dy$. Here we deal with bounded kernels $J$ but with 
polynomial tails, that is, we assume a lower bound of the form $J(x,y) \geq c_1|x-y|^{-(n + 2\sigma)}$, for  $|x - y| > c_2$. Our estimates takes 
the form $\|u(t)\|_{L^q(\R^n)} \leq C t^{-\frac{n}{2\sigma} (1 - \frac{1}{q})}$ for $t$ large.
\end{abstract}

\maketitle

\section{Introduction.}

Let $J(x,y)$ be a bounded, positive, continuous, symmetric function defined for $(x,y)\in \R^n \times \R^n$. Here we address nonlocal Cauchy problems
of the form
\begin{eqnarray}\label{eq}
\partial_t u(t,x) = \int_{\R^n} J(x,y)[u(t,y) - u(t,x)] dy \quad \mbox{in} \ \R_+ \times \R^n
\end{eqnarray}
with the initial condition $u(0, x) = u_0(x)$ satisfying $u_0 \in L^1(\R^n) \cap L^\infty(\R^n)$. On $J$ we also assume the basic condition
\begin{equation}\tag{J1}\label{integrabilityJ}
\begin{split}
x \mapsto J(x,y) \in L^1(\R^n) \quad \mbox{for all} \ y \in \R^n.
\end{split}
\end{equation}

By the symmetry of $J$ we have the same property exchanging $x$ and $y$.

Our main interest here is the asymptotic
behavior of the solution of~\eqref{eq} as $t \to +\infty$. It is well known that the diffusive nature of the equation implies that 
the solution goes to zero
when $t \to +\infty$. For smooth kernels $J$ with compact support, it is proven in \cite{Ignat-Rossi} that the solution $u$ of the equation~\eqref{eq}
has the decay estimate
$$
\|u(\cdot, t)\|_{L^q(\R^n)} \leq C t^{-\frac{n}{2}(1 - \frac{1}{q})}
$$
for any $q \in [1, \infty)$. Note that this decay rate is the same as the one that holds for solutions of the classical Heat equation. In the case of an equation in
convolution form, that is when $J(x,y) = K(x - y)$ with $K$ a nonnegative radial function, not necessarily compactly supported, it is proven
in \cite{Chasseigne-Chaves-Rossi} that the solutions of equations with the form~\eqref{eq} have the decay estimate
\begin{equation}\label{fractionaldecay}
\|u(\cdot, t)\|_{L^q(\R^n)} \leq C t^{-\frac{n}{2\sigma}(1 - \frac{1}{q})},
\end{equation}
provided the function $K$ has a Fourier transform satisfying the expansion $\hat{K}(\xi) = 1 - A|\xi|^{2\sigma} + o(|\xi|^{2\sigma})$, where $A > 0$
is a constant. It is remarkable that in this case the decay estimate is analogous to the one for the $\sigma-$order fractional heat equation, $v_t = -(-\Delta)^\sigma v$, 
with $\sigma \in (0,1)$. We also note that the convolution form of the equation allows the use of Fourier analysis to obtain this result.

In this work we consider kernels $J$ not in convolution form and not compactly supported in order to investigate the diffusive role of 
the asymptotic tails of the kernels when doing 
decay estimates. We illustrate
this feature asking for an asymptotic profile for $J$ with the following lower bound
\begin{equation}\tag{J2}\label{hypJ}
J(x,y) \geq c_1|x-y|^{-(n + 2\sigma)}, \quad \mbox{for } \ |x - y| > c_2
\end{equation}
for certain constants $c_1, c_2 > 0$ and $\sigma \in (0,1)$. For simplicity, to clarify the arguments we will assume $c_2 = 1$.

We remark that the use of Fourier analysis is not helpful here due to fact that our operator is not in convolution form.
Despite of this difficulty, energy methods can be carried out in order to obtain the main result of this paper, which reads as follows:

\begin{teo}\label{teo}
Let $ n \geq 2$ and $J$ be a bounded, nonnegative, continuous, symmetric function defined in $\R^n$ satisfying~\eqref{integrabilityJ}
and~\eqref{hypJ}.
Then, for each $q \in (2\sigma, +\infty)$, the solution of~\eqref{eq} associated to an initial condition $u_0 \in L^1(\R^n) \cap L^\infty(\R^n)$
goes to zero in $L^q (\R^n)$ as $t \to +\infty$. Moreover, the solution has an asymptotic decay that can be bounded by
\begin{eqnarray}\label{Lqdecay.11}
\limsup_{t\to \infty}  t^{\frac{n}{2\sigma} (1 - \frac{1}{q})} \|u(\cdot, t)\|_{L^q(\R^n)} \leq C \|u_{0}\|_{L^{1}(\R^{n})},
\end{eqnarray}
and, for all $t>0$, it holds
\begin{eqnarray}\label{Lqdecay}
\|u(\cdot, t)\|_{L^q(\R^n)} \leq C \max \{ \|u_{0}\|_{L^{1}(\R^{n})}, 
 \|u_{0}\|_{L^{q}(\R^{n})} \} t^{-\frac{n}{2\sigma} (1 - \frac{1}{q})},
\end{eqnarray}
where the constant $C$ depends on $q, \sigma$ and $n$.
\end{teo}

For the case $q \in (1, 2\sigma]$ we have the following 
\begin{cor}\label{cor}
Let $ n \geq 2$ and $J$ be a bounded, nonnegative, continuous, symmetric function defined in $\R^n$ satisfying~\eqref{integrabilityJ}
and~\eqref{hypJ}. Let $u$ be the solution of~\eqref{eq} associated to an initial condition $u_0 \in L^1(\R^n) \cap L^\infty(\R^n)$. 
If $q \in (1, 2\sigma]$, then the asymptotic decay given by~\eqref{Lqdecay.11} holds. Moreover, for all 
$r > 2\sigma$ and for all $t > 0$ we have
\begin{equation*}
\|u(\cdot, t)\|_{L^q(\R^n)} \leq C \max \{ \|u_{0}\|_{L^{1}(\R^{n})}, 
 \|u_{0}\|_{L^{r}(\R^{n})} \} t^{-\frac{n}{2\sigma} (1 - \frac{1}{q})},
\end{equation*}
with $C$ independent of $u_0$. 
\end{cor}

Note that the kernel $J$ defining the nonlocal operator in~\eqref{eq} shares, at one hand, the ``fat'' tails of the fractional
Laplacian of order $2\sigma$, and on the other, the integrability property of zero order operators.
However, the decay rate of the solution of problem~\eqref{eq} is as in the case of the fractional Laplacian
(see~\eqref{fractionaldecay}) and not as in the classical zero order case defined by compactly supported kernels. This shows
the diffusive effect of the fractional Laplacian is more related with the tails of the kernel, which allows to transport mass to infinity through
long jumps, than the highly intense, but shorter jumps related with the singularity at the origin.

\section{Basic Facts and Preliminaries.}

Concerning the existence and uniqueness of problem~\eqref{eq}, the symmetry, boundedness and integrability assumptions over $J$,
allows us to perform a fixed point argument to obtain the following result whose proof is omitted.
\begin{teo}[\cite{libro}, \cite{Ignat-Rossi}]\label{existence}
Let $J: \R^n \times \R^n \to \R$ be a bounded and symmetric function satisfying assumption~\eqref{integrabilityJ} and~\eqref{hypJ},
and let $u_0 \in L^1(\R^n) \cap L^\infty(\R^n)$. Then, there exists a unique solution
$u \in C([0,+\infty), L^1(\R^n) \cap L^\infty(\R^n))$ of equation~\eqref{eq}. This solution satisfies
$||u(\cdot, t)||_{L^1(\R^n)} \leq ||u_0||_{L^1(\R^n)}$ and $||u(\cdot, t)||_{L^\infty(\R^n)} \leq ||u_0||_{L^\infty(\R^n)}$ for all $t \geq 0$.
\end{teo}

The symmetry assumption over $J$ allows us to use an energy approach in order to get Theorem~\ref{teo}. Roughly speaking, this assumption
allows us to ``integrate by parts" equation~\eqref{eq}. For $q > 1$ we multiply the equation by $q|u|^{q-2}u$ and integrate, obtaining the inequality
\begin{equation}\label{testing0}
\begin{split}
& \partial_t \int_{\R^n} |u(x)|^q dx \\
= & -\frac{q}{2} \iint_{\R^{2n}} J(x,y) (u(y) - u(x))(|u(y)|^{q-2}u(y) - |u(x)|^{q-2}u(x))\, dy \, dx,
\end{split}
\end{equation}
where we have avoided the dependence on $t$ of the function $u$ for simplicity.

Now we recall the following inequality (whose proof is given in the Appendix): 
let $q > 1$ and $a, b \neq 0$. Then, there exists a constant $C$ depending only on $q$, such that
\begin{equation}\label{desigualdad1}
(a - b) (|a|^{q - 2}a - |b|^{q - 2}b) \geq C |a - b|^q.
\end{equation}

Hence, using this inequality into~\eqref{testing0}, we conclude
\begin{equation}\label{testingqbig}
\begin{split}
\partial_t ||u||_{L^q(\R^n)}^q \leq -C_q \iint_{\R^{2n}} J(x,y) |u(y) - u(x)|^q dy dx =: -C_q E_{J,q}(u).
\end{split}
\end{equation}

Our strategy is to get an estimate for the energy $E_{J,q}(u)$ in the right-hand side of the above expression in terms of a higher
$L^p-$norm of $u$.
This can be accomplished due to the asymptotic behavior~\eqref{hypJ} of the kernel. Being similar to the tails of the kernel of the fractional
Laplacian of order $2\sigma$, we compare $E_{J, q}$ in~\eqref{testingqbig} with an ad-hoc fractional seminorm, for which Sobolev-type inequalities
are available.

We recall that for $\sigma \in (0,1)$ and $q \in [1, \infty)$, $W^{\sigma, q}(\R^n)$ is the fractional Sobolev space of all $L^q(\R^n)$
functions with bounded fractional seminorm $[v]_{\sigma, q}$, given by
\begin{eqnarray}\label{defseminorm}
[v]_{\sigma, q}^q = \iint_{\R^{2n}} \frac{|v(x + z) - v(x)|^q}{|z|^{n + q\sigma}} dx\, dz.
\end{eqnarray}

Under these definitions, the key step in our approach is to relate the energy in~\eqref{testingqbig} with the corresponding fractional seminorm
defined in~\eqref{defseminorm}. Once this is accomplished, we apply the following fractional Sobolev-type inequality, which asserts the existence
of a constant $C > 0$ such that for each $v \in W^{\sigma, q}(\R^n)$ with $\sigma q < n$, it holds
\begin{eqnarray}\label{sobolevineq}
||v||_{L^{q^\star}(\R^n)}^q \leq C [v]_{\sigma, q}^q,
\end{eqnarray}
where $q^\star = q^\star(\sigma) = nq / (n - \sigma q)$ (see~\cite{Hitch}). By the use of the above inequality we get the desired increment of
the $L^p-$norm of the solutions, and after this the proof follows standard arguments.

Once we have obtained Theorem~\ref{teo} which is valid for $q > 2\sigma$, the corresponding decay estimate for the remaining case 
$q \in (1, 2\sigma]$ can be obtained by interpolation.
Hence, at the very end our arguments rely on estimates for the nonlinear operator in the right hand side of~\eqref{testingqbig}.

\section{Fractional Seminorm Estimates.}

First, we consider a positive smooth function $\psi: \R^n \to \R$ with the following properties
\begin{equation}\label{hyppsi}
\mbox{supp} (\psi) \subset B_1, \qquad \mbox{and} \qquad
\int_{\R^n} \psi(x)dx = 1.
\end{equation}

With the aid of this function, we split a function $u$ into two parts. We will denote the ``smooth'' part of $u$ as $v$ and the remaining as $w$. We let
\begin{equation}\label{defvw}
v(t,x) := \int_{\R^n} \psi(x - z) u(t, z) dz; \qquad u(t,x) := v(t,x) + w(t,x).
\end{equation}

As a first property of this decomposition we have that each $L^q$ norm of the functions $v$ and $w$ is controlled by the corresponding norm of $u$.
\begin{lema}\label{vw<u}	
Let $v$ and $w$ be given by~\eqref{defvw}. For each $q \in [1, +\infty)$, there exists $C = C(q, \psi)$ such that
\begin{equation*}
\|v\|_{L^q(\R^n)} \leq C \|u\|_{L^q(\R^n)}, \quad \mbox{and} \quad \|w\|_{L^q(\R^n)} \leq C \|u\|_{L^q(\R^n)}.
\end{equation*}
\end{lema}

\begin{proof} We start with $v$. Using its definition, we have
$$
\begin{array}{l}
\displaystyle
\int_{\R^n} |v(x)|^qdx = \int_{\R^n} \Big{|} \int_{\R^n} \psi(x- y) u(y) dy \Big{|}^q dx \\[10pt]
\qquad\displaystyle = \int_{\R^n} \Big{|} \int_{\R^n} \psi(x- y)^{1/q'} \psi(x-y)^{1/q} u(y) dy \Big{|}^q dx \\[10pt]
\qquad\displaystyle \leq \int_{\R^n} \Big{[} \Big{(} \int_{\R^n} \psi(x - y) dy\Big{)}^{1/q'} \Big{(}\int_{\R^n} \psi(x - y)|u(y)|^q dy\Big{)}^{1/q} \Big{]}^q dx \\[10pt]
\qquad\displaystyle
\leq C(q, \psi) \int_{\R^n} \int_{\R^n} \psi(x - y) |u(y)|^q dy dx \\[10pt]
\qquad\displaystyle
= C(q, \psi) \int_{\R^n}  |u(y)|^q \int_{\R^n} \psi(x - y) dx dy \\[10pt]
\displaystyle \qquad
\leq C(q, \psi) \int_{\R^n} |u(y)|^q dy.
\end{array}
$$
The inequality conrresponding to $w$ easily follows from the triangular inequality in $L^q$.
\end{proof}

Now we state a key result to get the desired estimate on the decay rate.
\begin{prop}\label{keyprop}
Let $J: \R^n \times \R^n \to \R_+$ a bounded and symmetric function satisfying hypotheses~\eqref{integrabilityJ} and~\eqref{hypJ}.
Let $\psi$ satisfying~\eqref{hyppsi} and $q \in (2\sigma, + \infty)$.
Then, there exists a constant $C > 0$ such that for all $u \in L^q(\R^n)$ and $v, w$ defined in~\eqref{defvw}, we have
\begin{equation}\label{deskeyprop}
[v]_{2\sigma q^{-1}, q}^q + \|w\|_{L^q(\R^n)}^q \leq C \iint_{\R^{2n}} J(x,y) |u(x) - u(y)|^q dx\, dy.
\end{equation}

The constant $C$ depends on $\psi, \sigma, q$ and $n$.
\end{prop}

\begin{proof} For the estimate concerning $w$, we have
\begin{eqnarray*}
\int_{\R^n} |w(x)|^q dx &=& \int_{\R^n} |u(x) - v(x)|^q dx \\
&=& \int_{\R^n} \Big{|} u(x) - \int_{\R^n} \psi(x - z)u(z)dz \Big{|}^q dx \\
&=& \int_{\R^n} \Big{|} \int_{\R^n} \psi(x - z) (u(x) - u(z))dz \Big{|}^q dx \\
&=& \int_{\R^n} \Big{|} \int_{\R^n} \psi(x - z)^{1/q'} \psi(x - z)^{1/q} (u(x) - u(z))dz \Big{|}^q dx.
\end{eqnarray*}

Applying Holder's inequality, we get
$$
\begin{array}{l}
\displaystyle \int_{\R^n} |w(x)|^q dx
\\[10pt]\displaystyle
\leq
\int_{\R^n} \Big{(} \int_{\R^n} \psi(x - z)dz \Big{)}^{q/q'} \Big{(} \int_{\R^n}\psi(x - z) |u(x) - u(z)|^q dz \Big{)} dx \\[10pt]\displaystyle
\leq C(q, q', \psi) \int_{\R^n} \int_{\R^n} \psi(x - z) |u(x) - u(z)|^q dz\, dx.
\end{array}
$$

Since $\psi$ is supported in $B_1$, clearly we have that for all $|x - z| \geq 1$, $\psi(x - z) \leq J(x, z)$. Meanwhile, since $J$ is positive, 
when $|x - z| < 1$ there exists a constant $C$ depending only on $|\psi|_\infty$ such that $\psi(x - z) \leq C J(x,z)$. Then
$$
\|w\|^q_{L^q(\R^n)} \leq C \iint_{\R^{2n}} J(x,y) |u(x) - u(y)|^q dx\, dy.
$$

Now we deal with the term concerning $v$. We split the fractional seminorm as
\begin{equation*}
\begin{split}
[v]_{2\sigma q^{-1}, q}^q
& = \ \iint_{|x - y| > 1} \frac{|v(x) - v(y)|^q}{|x - y|^{n + 2\sigma}}dxdy
+ \iint_{|x - y| \leq 1} \frac{|v(x) - v(y)|^q}{|x - y|^{n + 2\sigma}}dxdy \\
& =: \ I_{ext} + I_{int}.
\end{split}
\end{equation*}

We look at these integrals separately. For $I_{ext}$, using the definition of $v$ we have
\begin{equation*}
\begin{split}
I_{ext} = \iint_{|x - y| > 1} \Big{|} \int_{\R^n} (u(x - z) - u(y - z)) \psi(z)dz \Big{|}^q |x - y|^{-(n + 2\sigma)} dx dy.
\end{split}
\end{equation*}

Now, we can look at the measure $\mu(dz) = \psi(z)dz$ as a probability measure because of~\eqref{hyppsi}. Clearly, the function
$t \mapsto |t|^q$ is convex in $\R$. Then, we apply Jensen's inequality on the $dz-$integral in right-hand side of the last expression, concluding
\begin{equation*}
I_{ext} \leq \iint_{|x - y| > 1} \int_{\R^n} |u(x - z) - u(y - z)|^q \psi(z)dz |x - y|^{-(n + 2\sigma)} dx dy,
\end{equation*}
which, after an application of Fubini's Theorem, gives
\begin{equation*}
I_{ext} \leq  \int_{\R^n} \psi(z) \Big{(} \iint_{|x - y| > 1} |u(x - z) - u(y - z)|^q |x - y|^{-(n + 2\sigma)} dx dy \Big{)} dz,
\end{equation*}

Then, applying the change $\tilde{x} = x - z, \tilde{y} = y - z$ in the $dx dy$ integral and using~\eqref{hyppsi}, we conclude
\begin{equation*}
\begin{split}
I_{ext} \leq &
\int_{\R^n} \psi(z) \Big{(} \iint_{|\tilde{x} - \tilde{y}| > 1} |u(\tilde{x}) - u(\tilde{y})|^q
|\tilde{x} - \tilde{y}|^{-(n + 2\sigma)} d\tilde{x} d\tilde{y} \Big{)} dz  \\
= & \iint_{|\tilde{x} - \tilde{y}| > 1} |u(\tilde{x}) - u(\tilde{y})|^q
|\tilde{x} - \tilde{y}|^{-(n + 2\sigma)} d\tilde{x} d\tilde{y}.
\end{split}
\end{equation*}

Using this last expression, we obtain from the assumption~\eqref{hypJ} that
\begin{equation}\label{q=2-1}
I_{ext} \leq C \iint_{\R^{2n}} J(x,y) |u(x) - u(y)|^q dx\, dy.
\end{equation}

Now we deal with $I_{int}$. In this case, using the definition of $\psi$, we can write
\begin{equation}\label{q=2Iint}
I_{int} = \iint_{|x - y| < 1} \Big{|} \int_{\R^n} u(z) (\psi(x - z) - \psi(y - z)) dz \Big{|}^q |x - y|^{-(n + 2\sigma)} dx dy.
\end{equation}

Note that by using~\eqref{hyppsi}, we have for all $x, y \in \R^n$
\begin{equation*}
\int_{\R^n} u(x)(\psi(x - z) - \psi(y - z)) dz = u(x) \Big{(} \int_{\R^n} \psi(x - z)dz - \int_{\R^n} \psi(y - z)dz \Big{)} = 0,
\end{equation*}
and then
\begin{equation*}
\int_{\R^n} u(z) (\psi(x - z) - \psi(y - z)) dz = \int_{\R^n} (u(z) - u(x)) (\psi(x - z) - \psi(y - z)) dz.
\end{equation*}

Thus, using this equality into~\eqref{q=2Iint}, we get
\begin{equation*}
I_{int} = \iint_{|x - y| < 1} \Big{|} \int_{\R^n} (u(z) - u(x)) (\psi(x - z) - \psi(y - z)) dz \Big{|}^q |x - y|^{-(n + 2\sigma)} dx dy.
\end{equation*}

However, note that if $|x - z| \geq 2$ in the $dz$ integral, since $|x - y| < 1$ necessarily $|y - z| > 1$. Then, due to the fact that $\psi$ is
supported in the unit ball, the contribution of the integrand when $|x - z| \geq 2$ is null in the $dz$ integral. Taking this into account,
applying H\"older's inequality into the $dz-$integral, we have
\begin{equation*}
\begin{split}
& I_{int} \\
= & \iint_{|x - y| < 1} \Big{|} \int_{|x - z| < 2} (u(z) - u(x)) (\psi(x - z) - \psi(y - z)) dz \Big{|}^q \\
& \qquad \qquad \times
|x - y|^{-(n + 2\sigma)} dx dy \\
\leq & \iint_{|x - y| < 1} \Big{(} \int_{|x - z| < 2} |u(z) - u(x)|^q dz \Big{)} \\
& \qquad \qquad \times \Big{(}\int_{|x - \tilde{z}| < 2} |\psi(x - \tilde{z}) - \psi(y - \tilde{z})|^{q'} d\tilde{z} \Big{)}^{q/q'}
|x - y|^{-(n + 2\sigma)} dx dy,
\end{split}
\end{equation*}
where $q' = q/(q - 1)$ is the H\"older conjugate to $q$. By Fubini's Theorem we can write
\begin{equation*}
I_{int} = \int_{x \in \R^n} \Big{(} \int_{|x - z| < 2} (u(z) - u(x))^2 dz \Big{)} \Psi(x) dx,
\end{equation*}
where
\begin{equation*}
\Psi(x) = \int_{|x - y| < 1} \Big{(} \int_{|x - \tilde{z}| < 2} |\psi(x - \tilde{z}) - \psi(y - \tilde{z})|^{q'} d\tilde{z} \Big{)}^{q/q'}
|x - y|^{-(n + 2\sigma)} dy.
\end{equation*}

Using the regularity of $\psi$, we have
\begin{equation*}
\begin{split}
\Psi(x) \leq & \int_{|x - y| < 1} \Big{(} \int_{|x - \tilde{z}| < 2} ||D\psi||_\infty^{q'} |x - y|^{q'} d\tilde{z} \Big{)}^{q/q'}
|x - y|^{-(n + 2\sigma)} dy \\
\leq & ||D\psi||_\infty^{q} |B_2|^{q/q'} \int_{|x - y| < 1} |x - y|^q |x - y|^{-(n + 2\sigma)} dy,
\end{split}
\end{equation*}
and since $q > 2\sigma$, we conclude the last integral is convergent, obtaining
\begin{equation*}
\Psi(x) \leq C_{n, \sigma, q}||D\psi||_\infty^{q} |B_2|^{q/q'},
\end{equation*}
which leads us to the following estimate for $I_{int}$
\begin{equation*}
I_{int} \leq C \int_{x \in \R^n} \int_{|x - z| < 2} |u(z) - u(x)|^q dz dx.
\end{equation*}

From this, it is easy to get
\begin{equation*}
I_{int} \leq C \int_{|x - z| \leq 2} \frac{|u(z) - u(x)^q}{(1 + |x - z|)^{n + 2\sigma}} dz dx ,
\end{equation*}
which by the use of~\eqref{hypJ}, let us conclude that
\begin{equation*}
I_{int} \leq C \iint_{\R^{2n}} J(x,y) |u(x) - u(y)|^q dx\, dy.
\end{equation*}

This last estimate together with~\eqref{q=2-1} concludes the proof.
\end{proof}


\section{Proof of Theorem~\ref{teo}}

For simplicity, we use the following notation: for $q > 1$, we denote
$\tilde{\sigma} = \tilde{\sigma}_q = 2\sigma q^{-1}$.
Hence, the fractional critical exponent of $q$ relative to $\tilde{\sigma}$ is denoted by $\tilde{q}^\star$ and is given by
\begin{equation*}
\tilde{q}^\star = \frac{nq}{n - q \tilde{\sigma}} = \frac{nq}{n - 2\sigma}.
\end{equation*}

We need an intermediate result.
\begin{lema}\label{lemanormauLq}
Let us assume the hypotheses of Theorem \ref{existence} and Proposition~\ref{keyprop}. Then, there exists two constants $C_1, C_2 > 0$ depending on $n,q$ and $\sigma$ such that
\begin{equation}\label{ineqlemanormaLq}
||u(\cdot,t)||_{L^q(\R^n)}^q \leq C_1 \|u_{0}\|_{L^{1}(\R^{n})}^{q(1-\theta)} E_{J, q}(u)^\theta + C_2 E_{J, q}(u),
\end{equation}
where $E_{J,q}$ is defined in~\eqref{testingqbig} and $\theta \in (0,1)$ satisfies the equality
\begin{equation} \label{theta}
\frac{1}{q} = \frac{\theta}{\tilde{q}^\star} + (1 - \theta), \qquad \mbox{that is} \qquad
\theta = 1 - \frac{2\sigma}{n(q - 1) + 2\sigma}.
\end{equation}
\end{lema}

\noindent
{\bf Proof:} Using the definition of $v$ and $w$ in~\eqref{defvw}, we have
\begin{equation}\label{ineqlemanormaLq1}
||u||_{\Lq}^{q} \leq 2^{q-1}\big(||v||_{\Lq}^{q} + ||w||_{\Lq}^{q}\big)\,.
\end{equation}

By the definition of $v$, it belongs to $L^p$ for all $p$. Hence, we can interpolate, obtaining
\begin{equation*}
||v||_{\Lq} \leq ||v||_{L^{\tilde{q}^\star}(\R^n)}^\theta ||v||_{L^1(\R^n)}^{1 - \theta},
\end{equation*}
where $\theta$ is given by~\eqref{theta}. Recalling the Sobolev-type inequality~\eqref{sobolevineq}, we have
\begin{equation*}
||v||_{\Lq} \leq C[v]_{\tilde{\sigma}, q}^\theta ||v||_{L^1(\R^n)}^{1 - \theta},
\end{equation*}
and then, using this last inequality,~\eqref{ineqlemanormaLq1}, Proposition~\ref{keyprop} and the property
 $||u(\cdot, t)||_{L^1(\R^n)} \leq ||u_0||_{L^1(\R^n)}$ given in
Theorem~\ref{existence}, we get~\eqref{ineqlemanormaLq}.
\qed

\medskip
\noindent
{\bf Proof of Theorem~\ref{teo}.} Along this proof, we denote by $C_i$ a constant independent of $u_0$ and $K_i$ a constant depending on $u_0$
(making explicit this dependence). We recall that $q > 2\sigma$.
Consider the constants $C_1, C_2$ of the previous lemma and denote
\begin{equation}\label{K1yH}
K_1 = C_1 ||u_0||_{L^1(\R^n)}^{q(1 - \theta)}, \quad \mbox{and} \quad H(x) = K_1 x^\theta + C_2 x, \ \mbox{for} \ x \geq 0.
\end{equation}

With these definitions, inequality~\eqref{ineqlemanormaLq} can be written as
\begin{equation*}
||u||_{\Lq}^q \leq H(E_{J, q}(u)),
\end{equation*}
or equivalently
\begin{equation*}
H^{-1}(||u||_{\Lq}^q) \leq E_{J, q}(u).
\end{equation*}

Hence, denoting $\phi(t)=\phi[u](t):= ||u(\cdot, t)||_{\Lq}^q$, we use~\eqref{testingqbig} to conclude $\phi$ satisfies the differential inequality
\begin{equation}\label{diffineq0}
\partial_t \phi(t) + C_q H^{-1}(\phi(t)) \leq 0, \quad \mbox{for} \ t > 0.
\end{equation}

Direct computations allows us to write
\begin{equation*}
H^{-1}(x) \geq  \left \{ \begin{array}{ll} x/(2C_2) & \quad x \geq K_2, \\ 
(x/(2 K_1))^{1/\theta} & \quad x < K_2, \end{array} \right. \,
\end{equation*}
where $K_2 = C_3 ||u_0||_{L^1(\R^n)}^q$ for some constant $C_3$. Using this expression, we obtain from~\eqref{diffineq0} that, for all $t > 0$, 
$\phi(t)$ satisfies
\begin{equation}\label{ineqteoCaseI}
\partial_t \phi(t) + C \Big{(} \phi(t) \ \mathbf{1}_{\{ \phi(t) \geq K_2 \}} 
+ ||u_0||_{L^1(\R^n)}^{q(1 - \theta^{-1})}\phi(t)^{1/\theta} \mathbf{1}_{\{ \phi(t) < K_2 \}} \Big{)} \leq 0,
\end{equation}

We claim that there exists a $t_0 > 0$ such that $\phi(t_0) < K_2$. Otherwise, if $\phi(t) \geq K_2$, from~\eqref{ineqteoCaseI} we have
$
\partial_t \phi(t) + C \phi(t) \leq 0
$
holds for all $t > 0$, implying that $\phi$ has exponential decay, which is a contradiction with $\phi(t)\geq K_2>0$.
Hence, since there exists such $t_0$ and since \eqref{testingqbig}
implies in particular that $\phi$ is nonincreasing,
from~\eqref{ineqteoCaseI} we conclude that for all $t > t_0$,
\begin{equation*}
\partial_t \phi(t) + C \|u_{0}\|_{L^{1}(\R^{n})}^{(1-\theta)/\theta} (\phi(t))^{1/\theta} \leq 0\,.
\end{equation*}

From this we get
$$
\phi(t) \leq C \|u_{0}\|_{L^{1}(\R^{n})} t^{\theta/(\theta - 1)}\,,
$$
again for $C$ depending on $n,q$ and $\sigma$.
This proves that the $L^q$-norm of the solution goes to zero with the desired rate, that is, \eqref{Lqdecay.11} holds.

Now we look for a bound valid for all $t>0$, that is~\eqref{Lqdecay}. If $\phi(0) \leq K_2$, that is when
$||u_0||_{\Lq} \leq C_3 ||u_0||_{L^1(\R^n)}$, by~\eqref{ineqteoCaseI} and since $\phi$ is nonincreasing, we get
\begin{equation}\label{diffineq1}
\partial_t \phi(t) + C ||u_0||_{L^1(\R^n)}^{q(1 - \theta^{-1})}\phi(t)^{1/\theta} \leq 0 
\end{equation}
for all $t > 0$. Integrating from $0$ to $t$, using the expression~\eqref{theta} for $\theta$ and recalling $\phi(t) = ||u(t)||_{\Lq}^q$, we conclude
\begin{equation*}
||u(\cdot, t)||_{\Lq} \leq C ||u_0||_{L^1(\R^n)} t^{- \frac{n}{2\sigma}(1 - q^{-1})},
\end{equation*}
where the constant $C$ does not depend on $u_0$.

In the case $\phi(0) > K_2$, the corresponding inequality $\phi(0) > C_3 ||u_0||_{L^1(\R^n)}^q$ allows us to replace the dependence on 
$||u_0||_{L^1(\R^n)}$ of $H$ in~\eqref{K1yH} by the norm $||u_0||_{\Lq}$. Thus, this replacement makes the analogous inequality 
to~\eqref{diffineq0} that in this case reads as
\begin{equation}\label{diffineq2}
\partial_t \phi(t) + C_q \tilde{H}^{-1}(\phi(t)) \leq 0, \quad \mbox{for} \ t > 0. 
\end{equation}
where, for some $C_3 > 0$, we have denoted
\begin{equation*}
\tilde{H}(x) = C_3 ||u_0||_{\Lq}^{q(1 - \theta)} x^\theta + C_2 x.
\end{equation*}

Since $\phi$ is nondecreasing our interest is to estimate $\tilde{H}^{-1}(x)$ in~\eqref{diffineq2} from below in the interval $x \in (0, \phi(0))$.
It is easy to see that for all $C$ small enough, independent of $u_0$ we can get
\begin{equation*}
\tilde{H}(C ||u_0||_{\Lq}^{q(1 - \theta^{-1})} x^{1/\theta}) \leq x, \quad \mbox{for all} \ x \in (0,\phi(0)),
\end{equation*}
which means that
\begin{equation*}
C ||u_0||_{\Lq}^{q(1 - \theta^{-1})} x^{1/\theta} \leq \tilde{H}(x), \quad \mbox{for all} \ x \in (0,\phi(0)).
\end{equation*}

Replacing this inequality into~\eqref{diffineq2} we obtain that
\begin{equation*}
\partial_t \phi(t) + C ||u_0||_{\Lq}^{q(1 - \theta^{-1})} \phi(t)^{1/\theta} \leq 0
\end{equation*}
holds for all $t > 0$, where $C> 0$ is independent of $u_0$. This inequality is exactly the same as~\eqref{diffineq1} but with the $L^q-$norm 
of $u_0$ replacing the $L^1-$norm. Integrating, we conclude the estimate
\begin{equation*}
||u(\cdot, t)||_{\Lq} \leq C ||u_0||_{L^q(\R^n)} t^{- \frac{n}{2\sigma}(1 - q^{-1})},
\end{equation*}
and therefore inequality~\eqref{Lqdecay} holds.
\qed

\medskip
\noindent
{\bf Proof of Corollary~\ref{cor}:}
By an interpolation argument, there exists $\theta \in (0,1)$  such that
\begin{equation*}
\begin{split}
||u(\cdot, t)||_{\Lq} \leq & ||u(\cdot, t)||_{L^1(\R^n)}^\theta ||u(\cdot, t)||_{L^r(\R^n)}^{1 - \theta} \\
\leq & C ||u_0||_{L^1(\R^n)}^\theta \max \{ \|u_{0}\|_{L^{1}(\R^{n})}, 
\|u_{0}\|_{L^{r}(\R^{n})} \}^{1 - \theta} t^{-\frac{n}{2\sigma} (1 - \frac{1}{q})},
\end{split}
\end{equation*}
where we have used Theorem~\ref{teo} in the last inequality.
\qed

\section{Appendix}
\noindent
{\bf Proof of Inequality~\eqref{desigualdad1}:} Clearly the inequality holds if $a = b$, then, without loss of generality, we assume $|a| > |b|$.
If this is the case,~\eqref{desigualdad1} is equivalent to
\begin{equation*}
(1 - x)(1 - |x|^{q-2}x) \geq C|1 - x|^q
\end{equation*}
for all $|x| < 1$. Now, in the case $x < 0$ we have the left-hand side is bounded from above by 1, meanwhile the right-hand side is bounded from
above by $2^q$. Hence, it is sufficient to consider $C = 2^{-q}$. For the case $x \in (0,1)$, the desired inequality is equivalent to
\begin{equation*}
1 \geq x^{q - 1} + C(1 - x)^{q-1}.
\end{equation*}

When $q \geq 2$, this inequality holds with any $C \leq 1$ because
$$
1 \geq x + (1 - x) \geq x^{q - 1} + (1 - x)^{q - 1}.
$$

When $q \in (1, 2)$, the conclusion follows from the fact that the derivative of $x \mapsto x^{q - 1} + C(1 - x)^{q-1}$ is strictly positive 
(independently of $C$) at $x = 1$. Then, choosing $C > 0$ small enough we obtain the result.
\qed

\medskip

\noindent
{\bf Aknowledgements:} J.D.R. was partially supported by MEC MTM2010-18128 and MTM2011-27998 (Spain) and FONDECYT 1110291 International Coperation. 
P.F. was partially supported by 
FONDECYT 1110291 and Programa Basal-CMM U. de Chile.
E.T. was partially supported by CONICYT, Grants Capital Humano Avanzado, Realizaci\'on de Tesis Doctoral and Cotutela en el 
Extranjero.

\end{document}